\documentclass[a4paper, 12pt]{amsart}
\usepackage{amssymb}
\usepackage{amsthm}
\usepackage{amsmath}
\usepackage{tikz}
\usepackage{graphicx}
\usepackage{caption}
\usepackage{subcaption}
\usepackage{diagbox}
\usepackage{url}
\usepackage[utf8]{inputenc}
\usepackage[english]{babel}
\usepackage[shortlabels]{enumitem}
\usepackage{comment}
\usepackage[table]{xcolor}

\DeclareMathOperator{\HS}{\mathrm{HS}}

\DeclareMathOperator{\rank}{\mathrm{rank}}
\DeclareMathOperator{\sgn}{\mathrm{sgn}}

\newtheorem{theorem}{Theorem}[section]
\newtheorem{corollary}[theorem]{Corollary}
\newtheorem{lemma}[theorem]{Lemma}
\newtheorem{conjecture}[theorem]{Conjecture}
\newtheorem{proposition}[theorem]{Proposition}
\newtheorem{question}[theorem]{Question}

\theoremstyle{definition}
\newtheorem{definition}[theorem]{Definition}
\newtheorem{example}[theorem]{Example}
\newtheorem{remark}[theorem]{Remark}

\title[Max rank properties for symmetric polynomials in a monomial CI]{On maximal rank properties for symmetric polynomials in an equigenerated monomial complete intersection}

\author{Filip Jonsson Kling and Samuel Lundqvist}

 \address{Filip Jonsson Kling, Department of Mathematics, Stockholm University, 106 91 Stockholm, 
 Sweden}
 \email{filip.jonsson.kling@math.su.se}

 \address{Samuel Lundqvist, Department of Mathematics, Stockholm University, 106 91 Stockholm, 
 Sweden}
 \email{samuel@math.su.se}

\begin{document}

\begin{abstract}
It is well known that a monomial complete intersection has the strong Lefschetz property in characteristic zero. This property is equivalent to the statement that any power of the sum of the variables is a maximal rank element on the complete intersection. In this paper, we investigate what happens when this element is replaced by another symmetric polynomial, in an equigenerated complete intersection.

We answer the question completely for the power sum symmetric polynomial using a grading technique, and for any Schur polynomial in the case of two variables by deriving a closed formula for the determinants of a family of Toeplitz matrices. Further, we obtain partial results in three or more variables for the elementary and the complete homogeneous symmetric polynomials and pose several open questions.
\end{abstract}

\maketitle

\section{Introduction}

The Lefschetz properties have recently been a popular area of study in commutative algebra and related fields. 
If $A$ is a graded Artinian algebra, then $A$ has the weak Lefschetz property if there is a linear form $\ell\in A_1$ such that the multiplication map
\[
\cdot \ell: A_i \to A_{i+1}
\]
has maximal rank, that is, is injective or surjective, for all $i$. If also all powers of $\ell$ give multiplication maps that have maximal rank, we say that $A$ has the strong Lefschetz property. The foundational example of an algebra satisfying the strong Lefschetz property is the Artinian monomial complete intersection $\mathbf{k}[x_1,\dots, x_n]/(x_1^{d_1},\dots, x_n^{d_n})$ over a field $\mathbf{k}$ of characteristic zero, proven independently by Stanley \cite{Stanley1980WeylGT} and Watanabe \cite{WATANABE1989194}. In their work, they show that multiplication by the element $x_1+\cdots + x_n$ and its powers always has maximal rank. A natural question is then if multiplication by other polynomials also give maps that have maximal rank. 

The polynomial $x_1+\cdots + x_n$ is the first element of many well studied families of symmetric polynomials. It can be seen as $p_1$, a power sum symmetric polynomial, or $h_1$, a complete symmetric polynomial, or $e_1$, an elementary symmetric polynomial, or as $s_{(1,0,\dots, 0)}$, a Schur polynomial. All of these families of symmetric polynomials have members of higher degrees that are different from just powers of $x_1+\cdots + x_n$. 

The main goal of this article is 
to initiate the study of which of these symmetric polynomials give maps that have maximal rank on which monomial complete intersections, and which that do not. We restrict to the equigenerated case $d=d_1=\ldots = d_n$. Our study connects to previous results and opens questions about algebras fixed by the action of the symmetric group $S_n$. 

Conca, Krattenthaler, and Watanabe \cite{Conca2009} asked when the ideals generated by $n$ power sum symmetric polynomials, or by $n$ complete symmetric polynomials, defines complete intersections. Up to this day, even the case $n=3$ remains unclassified. 

Haglund, Rhoades, and Shimozono \cite{HAGLUND2018851} studied the ideals $(x_1^d,\ldots,x_n^d) + I_{n,d}$ in connection with the Delta conjecture, where $I_{n,d}$ is generated by the elementary symmetric polynomials of degree $n-d+1,n-d+2,\ldots,n$, and managed to determine their Gr\"obner bases, providing a connection to Demazure characters. 

Kustin, Rahmati, and Vraciu \cite{KUSTIN2012256} studied the resolution of $(x_1^d,x_2^d,x_3^d)$ in the hypersurface ring $k[x_1,x_2,x_3]/(x_1^N + x_2^N + x_3^N)$, and Kustin, R.G. and Vraciu \cite{MR4332033,MR4506076} considered the corresponding question for $n = 4$. There are also results in characteristic $p$, where  Miller, Rahmati, and R. G. \cite{Miller-Rahmati-RG} studied the structure of the Betti numbers for $n = 3$. 

Maximal rank questions for symmetric polynomials have also been considered before, where for example Moreno-Soc{\'i}as and Snellman \cite{Snellman2000SomeCA} showed that any even degree complete symmetric polynomial is a maximal rank element in the exterior algebra. 

The 
paper is structured as follows. In Section $2$ we give some preliminaries on the algebraic tools used. Section $3$ then concerns the power sum symmetric polynomials. Using a finer grading on the polynomials involved, we give a full characterization for when they define maximal rank elements in Theorem \ref{thm:most_failures}. Next, in Section $4$, we examine Schur polynomials. Here we determine a closed formula for the determinants of a family of Toeplitz matrices in Theorem \ref{thm:Toeplitz_det_formula} and use that to say when Schur polynomials in two variables are maximal rank elements in Theorem \ref{thm:Schur_n=2}. Some results on Schur polynomials in more variables are also established. In Section $5$ and $6$, we give conjectures for when the elementary and the complete symmetric polynomials of degree $d$ are maximal rank elements, and use Macaulay's inverse system to establish partial results. 
Then, in Section $7$, we collect some further observations and conjectures for future work. Finally, as a resource for the more combinatorially inclined reader, at the end we include an appendix on Artinian Gorenstein algebras with results on multiplications in such algebras that we were unable to locate elsewhere.

\section{Algebraic preliminaries}

Throughout this paper $\mathbf{k}$ will denote a field of characteristic zero. We will consider standard graded Artinian monomial complete intersections
$$A=\mathbf{k}[x_1,\dots, x_n]/(x_1^{d},\dots, x_n^{d}).$$
The Hilbert series of a graded algebra $R$, denoted $\HS(R;t)$, is the formal powers series that encodes the dimension of its graded pieces, so $\HS(R;t) = \sum_i \dim_{\mathbf{k}} (R_i) t^i$. It is well known that 
$$\HS(A;t)=\frac{(1-t^{d})^n}{(1-t)^n} = (1+t+\cdots+t^{d-1})^n.$$ In particular, the Hilbert series of $A$ is a polynomial of degree $n(d-1)$, called the socle degree of $A$, which is symmetric along $(n(d-1))/2$. 

We say that a form $f \in A_k$ is a \emph{maximal rank element}, or max-rank element for short, on $A$ if for any $i$, the multiplication map $\cdot f: A_i \to A_{i+k}$ is either injective or surjective. Equivalently, $f$ is a max-rank element on $A$ if the Hilbert series of $A/(f)$ equals \[[(1-t^k) \HS(A;t)],\] 
where the brackets means truncate at the first non-positive entry. 

Since $A$ is Gorenstein, an element $f \in A_k$ is a maximal rank element if and only if the multiplication by $f$ map from $A_j$ to $A_{j+k}$, for the largest $j$ such that $\dim_{\mathbf{k}} A_j \leq \dim_{\mathbf{k}} A_{j+k}$, is injective, see Lemma \ref{lem:inj_sufficient}. This is a property that will be used throughout the paper. Moreover, the difference 
\[ \HS(A/(f);t) - [(1-t^k) \HS(A;t)]  \] is symmetric, see  Proposition \ref{prop:Failing_symmetric}.

Macaulay's inverse system will be of use to determine vector space dimensions in Section $5$ and Section $6$. Let $R=\mathbf{k}[x_1,\ldots,x_n]$ and consider the dual $S=k[X_1,\ldots,X_n]$, where $x_i$ acts on $F \in S$ as $x_i \circ  F = \frac{\partial F}{\partial X_i}$. For an ideal $I=(f_1,\ldots,f_r) \subset R$, the inverse system, denoted $I^{-1}$, is the submodule annihilated by $f_1,\ldots,f_r$ under this action, and 
\[
\dim_{\mathbf{k}} (I^{-1})_d =
\dim_{\mathbf{k}} (R/I)_d. \]

\section{The power sum symmetric polynomial}

The goal for this section is to give a full classification for when the power sum symmetric polynomial
\[
p_{k,n} = x_1^k  + \cdots + x_n^k
\]
is a max-rank element on  
\[A=\mathbf{k}[x_1,\dots, x_n]/(x_1^d,\dots, x_n^d).\] As a first observation we note that $p_{k,n}=0$ in $A$ if $k\geq d$. Hence $p_{k,n}$ will fail to be a max-rank element for trivial reasons when $d\leq k \leq n(d-1)$, and will be a max-rank element for trivial reasons when $k>n(d-1)$, since all maps are expected to be the zero map only in the latter case. Therefore we may assume that $k<d$ for the remainder of this section. 

A crucial part in our argument is that we can use a finer grading on our monomials, as previously considered in \cite{BACKELIN20153158,MR4506076}. 

\begin{definition}
A monomial $x_1^{j_1}\cdots x_n^{j_n}$ has $k$\emph{-degree} $(d;a_1,a_2,\dots, a_n)$ if $d=\sum_{i=1}^nj_i$, $a_i\equiv j_i \pmod{k}$ and $0\leq a_i<k$ for all $i$. 
\end{definition}

The important thing to note about this grading for our purposes is that multiplication by $p_{k,n}$ respects it in the following sense. Let $\mathbf{a}=(a_1,a_2,\dots, a_n)$ and $m$ be a monomial of $k$-degree $(i;\mathbf{a})$. Then $p_{k,n}\cdot m$ has $k$-degree $(i+k;\mathbf{a})$. Hence the map $\cdot p_{n,k}: A_i \to A_{i+k}$ splits up as a direct sum
\[
\cdot p_{n,k}: \bigoplus_{a_1=0}^{k-1}\cdots \bigoplus_{a_n=0}^{k-1} A_{i;a_1,\dots, a_n} \to \bigoplus_{a_1=0}^{k-1}\cdots \bigoplus_{a_n=0}^{k-1} A_{i+k;a_1,\dots, a_n}
\]
where 
\[
\left(\cdot p_{n,k}\right) \left(A_{i;a_1,\dots, a_n} \right) \subseteq A_{i+k;a_1,\dots, a_n}.
\]
Therefore, when checking if $\cdot p_{n,k}: A_i \to A_{i+k}$ is injective, it is equivalent to checking if all maps $\cdot p_{n,k} :A_{i;a_1,\dots, a_n} \to A_{i+k;a_1,\dots, a_n}$ are injective.

Next, when considering an element of $A_{i;a_1,\dots, a_n}$, we may write it as $x_1^{a_1}\cdots x_n^{a_n}\cdot f(x_1^k,\dots, x_n^k)$ for some polynomial $f$. By changing the algebra we work in, it suffices to study elements of the form $f(x_1^k,\dots, x_n^k)$.

\begin{lemma}\label{lem:first_reduction}
An element $x_1^{a_1}\cdots x_n^{a_n}\cdot f(x_1^k,\dots, x_n^k)$ is a non-zero element in the kernel of multiplication by $p_{n,k}$ in $A=\mathbf{k}[x_1,\dots, x_n]/(x_1^d,\dots, x_n^d)$ if and only if $f(x_1^k,\dots, x_n^k)$ is a non-zero element in the kernel of multiplication by $p_{n,k}$ in $B=\mathbf{k}[x_1,\dots, x_n]/(x_1^{d-a_1},\dots, x_n^{d-a_n})$.
\end{lemma}

\begin{proof}
The lemma follows by repeated applications of the fact that $$x_l\cdot g \in (x_1^{s_1},\dots, x_n^{s_n})$$ for some integers $l,s_1,\dots, s_n$ and polynomial $g$ if and only if $$g\in (x_1^{s_1},\dots, x_{l-1}^{s_{l-1}}, x_{l}^{s_{l}-1}, x_{l+1}^{s_{l+1}}, \dots, x_n^{s_n}).$$ Taking $g=p_{n,k}\cdot f(x_1^k,\dots, x_n^k)$ then gives the equivalence of being a kernel element and $g=f(x_1^k,\dots, x_n^k)$ gives the equivalence of being non-zero.
\end{proof}

For the next step in our chain of reductions, it is again indirectly the finer grading that will help us out.

\begin{lemma}\label{lem:only_mult_of_k_matters}
Write $d=km+i$ for some $1\leq i \leq k$ and assume without loss of generality that $a_1,\dots, a_j<i$ while $a_{j+1},\dots, a_n\geq i$. Then $f(x_1^k,\dots, x_n^k)$ is a non-zero element in the kernel of multiplication by $p_{n,k}$ in the algebra $B=\mathbf{k}[x_1,\dots, x_n]/(x_1^{d-a_1},\dots, x_n^{d-a_n})$ if and only if it is a non-zero element in the kernel of multiplication by $p_{n,k}$ in $C=\mathbf{k}[x_1,\dots, x_n]/(x_1^{km+k},\dots, x_j^{km+k}, x_{j+1}^{km},\dots, x_n^{km})$. 
\end{lemma}

\begin{proof}
Let $g$ be a polynomial with $k$-degree $(s;0,0,\dots, 0)$, $I$ the defining ideal for $B$ and $J$ the defining ideal for $C$. We want to show that $g\in I$ if and only if $g\in J$. Since $J\subseteq I$, one implication is immediate. For the other direction, assume that $g\in I$. This means that every monomial in $g$ is divisible by $x_l^{d-a_l}$ for some $l$. We now note that since $g$ has $k$-degree $(s;0,0,\dots, 0)$, the exponent of every variable in a monomial in $g$ is divisible by $k$. Hence, if a variable is divisible by $x_l^{d-a_l}$, its exponent must be at least the smallest multiple of $k$ larger than or equal to $d-a_l$, which is $km+k$ if $l\leq j$ and $km$ otherwise. Thus $g\in J$. Finally, taking $g=f(x_1^k,\dots, x_n^k)$ and $g=p_{n,k}\cdot f(x_1^k,\dots, x_n^k)$ gives the statement of the lemma.
\end{proof}

For the next result, we need the function
\begin{equation}\label{eq:H-function}
H(n,k,i,j,m) = 
\begin{cases}
nk+2ij-ni-n-kj-k & \text{ if } mn-(n-j) \text{ is even}\\
nk+2ij-ni-n-kj-2k & \text{ if } mn-(n-j) \text{ is odd}.
\end{cases}
\end{equation}

With these preparations, we are now ready to prove the main result for determining when $p_{k,n}$ is a maximal rank element on $\mathbf{k}[x_1,\dots, x_n]/(x_1^d,\dots, x_n^d)$.

\begin{proposition}\label{prop:Inj_char}
Write $d=km+i$ for some $1\leq i \leq k$ and let $H$ be as defined in \eqref{eq:H-function}. Then $p_{k,n}$ is \emph{not} a max-rank element on $A=\mathbf{k}[x_1,\dots, x_n]/(x_1^d,\dots, x_n^d)$ if and only if there is a $j$ in the interval $0\leq j \leq n$ such that
\[
H(n,k,i,j,m)\geq 0.
\]
\end{proposition}

\begin{proof}
Since $A$ is a complete intersection, and therefore also Gorenstein, we know from Lemma \ref{lem:inj_sufficient} that it suffices to check if there are any unexpected kernel elements for multiplication by $p_{k,n}$. By using the $k$-degree considerations, we may assume that such an element has the form $x_1^{a_1}\cdots x_n^{a_n}\cdot f(x_1^k,\dots, x_n^k)$ for some polynomial $f$. If we assume that $a_1,\dots, a_j<i$ while $a_{j+1},\dots, a_n\geq i$, then Lemma \ref{lem:first_reduction} and Lemma \ref{lem:only_mult_of_k_matters} gives that we may equivalently look for kernel elements $f(x_1^k,\dots, x_n^k)$ in $C=\mathbf{k}[x_1,\dots, x_n]/(x_1^{km+k},\dots, x_j^{km+k}, x_{j+1}^{km},\dots, x_n^{km})$. For a final reduction, we note that $p_{k,n}\cdot f(x_1^k,\dots, x_n^k)$ being zero in $C$ is a statement where we may view $x_1^k,\dots, x_n^k$ as variables. Hence, if we set $y_t=x_t^k$ for all $1\leq t \leq n$, then this is equivalent to $(y_1+\cdots + y_n)\cdot f(y_1,\dots, y_n)$ being zero in $D=\mathbf{k}[y_1,\dots, y_n]/(y_1^{m+1},\dots, y_j^{m+1}, y_{j+1}^{m},\dots, y_n^{m})$. Since $y_1+\cdots + y_n$ is the sum of the variables in the monomial complete intersection $D$, the fact that $D$ has the weak Lefschetz property gives that such a polynomial $f$ appears in a degree $e$ if and only if $\dim_{\mathbf{k}}(D_e)>\dim_{\mathbf{k}}(D_{e+1})$. To determine the lowest such degree $e$, we first calculate the socle degree $s_D$ of $D$ to be
\begin{align*}
s_D&= j\left(m+1 \right) + (n-j)m - n \\
&=mn -(n-j).
\end{align*}
If $s_D$ is even, then we know from the shape of the Hilbert series of $D$ that the lowest possible value for $e$ is $s_D/2$, while it is $\frac{s_D+1}{2}$ if $s_D$ is odd, see \cite{RRR}. So pick an $f$ in the kernel of this degree. Then $f(x_1^k,\dots, x_n^k)$ has degree $ks_D/2$ or $k(s_D+1)/2$ and the lowest degree for $x_1^{a_1}\cdots x_n^{a_n}\cdot f(x_1^k,\dots, x_n^k)$ is therefore $ks_D/2+\sum a_s$ or $k(s_D+1)/2 + \sum a_s$, depending on if $mn -(n-j)$ is even or odd. 

Having found the lowest possible degree of a kernel element, we now need to see if injectivity is expected in that degree or not. Again, from our knowledge of the shape of Hilbert series of $A$, we know that multiplication by a form on $A$ is expected to be injective if the sum of the source degree and the target degree is at most the socle degree. In this case we have the socle degree $s_A=n(d-1) = n(km+i-1)$. When $mn -(n-j)$ is even, we also have the source degree $ks_D/2+\sum a_s$ and target degree $ks_D/2+\sum a_s + k$. Comparing their sum with $s_A$, we find
\begin{equation}\label{eq:inj_equation}
\begin{split}
&(ks_D/2+\sum a_s) + (ks_D/2+\sum a_s + k)\\ 
&= kmn -k(n-j) + 2\left(\sum_{s=1}^na_s \right) + k \\
&=s_A - ni -kn + kj + 2\left(\sum_{s=1}^na_s \right) + k + n.
\end{split}
\end{equation}
Since we are only interested in determining if there is a degree when injectivity is required, we only care about when $\sum_{s=1}^na_s$ is as low as possible, which is when $a_1=\cdots = a_j=0$ and $a_{j+1}=\cdots = a_n=i$. Then $\sum_s a_s = (n-j)i = ni-ji$. In those cases, \eqref{eq:inj_equation} simplifies to 
\begin{align*}
s_A - ni -kn + kj + 2\left(\sum_{s=1}^na_s \right) + k + n &= s_A + ni + kj - kn -2ij + k + n\\ 
&= s_A -H(n,k,i,j,m)
\end{align*}
and thus injectivity is required exactly when $H(n,k,i,j,m)\geq 0$. Finally, a similar computation with $k(s_D+1)/2 + \sum a_s$ in the case when $mn-(n-j)$ is odd also gives required injectivity when $H(n,k,i,j,m)\geq 0$ as desired.
\end{proof}

\begin{remark}
The arguments in the proof of Proposition \ref{prop:Inj_char} and the preceding discussion works equally as well when considering powers $(p_{k,n})^t$ of power sum symmetric polynomials since the algebra $D$ has the strong Lefschetz property. Assuming $k<d$ and $kt\leq n(d-1)$ to avoid trivial edge cases, given the same setup as in Proposition \ref{prop:Inj_char}, we have that $(p_{k,n})^t$ is not a max-rank element on $A$ if and only if there is a $j$ in the interval $0\leq j \leq n$ such that
\begin{align*}
nk+2ij-ni-n-kj-k\geq 0 & \text{ if } mn-(n-j)+t \text{ is odd}\\
nk+2ij-ni-n-kj-2k\geq 0 & \text{ if } mn-(n-j)+t \text{ is even}.
\end{align*}
\end{remark}

Before we give the full classification for when $p_{n,k}$ is a max-rank element, let us apply the above proposition in some specific instances, taking care of some sporadic cases when either $n$ or $k$ is small.

\begin{lemma}
\label{lem:n=2}
The form $p_{k,2}$ fails to have maximal rank on $\mathbf{k}[x,y]/(x^d,y^d)$ exactly in the following cases. If $k=2m$, then it fails in all cases except when $d\equiv m \pmod{k}$, and if $k=2m+1$, then it fails in all cases except when $d\equiv m \pmod{k}$ or $d\equiv m+1 \pmod{k}$. In particular, when $k=2$, it fails when $d$ is even, when $k=3$, it fails when $d$ is divisible by three, and when $k=4$, it fails unless $d\equiv 2 \pmod{4}$.
\end{lemma}
\begin{proof}
We know from Proposition \ref{prop:Inj_char} that we need to examine the non-negativity of $H(2,k,i,0,m)$, $H(2,k,i,1,m)$ and $H(2,k,i,2,m)$. Since $n=2$ is even, the parity of $m$ plays no role. Hence $H(2,k,i,0,m)=k-2i-2$, $H(2,k,i,1,m)=-2-k$ and $H(2,k,i,2,m)=2i-k-2$. Since $-2-k<0$ for all $k\geq 1$, we have failure if one of $k\geq 2(i+1)$ or $2i-2\geq k$ is satisfied. When $k=2m$, these reduce to $m\geq i+1$ and $i-1\geq m$, so $i=m$ is the only case where we can not establish failure. Similarly, when $k=2m+1$, the conditions reduce to $m+1/2\geq i+1$ and $m+1/2\leq i-1$, which gives that all cases fails to have maximal rank unless $i=m$ or $i+m+1$, as desired.
\end{proof}

\begin{lemma}
\label{lem:k=2}
The form $p_{2,n}$ fails expected injectivity on $\mathbf{k}[x_1,\dots, x_n]/(x_1^d,\dots, x_n^d)$ for any even $d>2$ when $n\geq 4$, and for any $d\equiv 2 \pmod{4}$ when $n=3$. In the remaining cases when $n\geq 3$, it has maximal rank.
\end{lemma}

\begin{proof}
First, write $d=2m+2$ and assume $n\geq 4$. By Proposition \ref{prop:Inj_char}, we get the desired failure of injectivity when $H(n,2,2,j,m)$ is non-negative for appropriate $j$ since $k=i=2$ in this case. Here
\[
H(n,2,2,j,m) =
\begin{cases}
2j-n-2 & \text{ if } mn - (n-j) \text{ is even} \\
2j-n-4 & \text{ if } mn - (n-j) \text{ is odd}. 
\end{cases}
\]
This quantity is maximized for large $j$, so inserting $j=n-1$ and $j=n$ gives that $H(n,2,2,n-1,m)$ and $H(n,2,2,n,m)$ attains the two values $n-4$ and $n-2$. Since one of them will always be non-negative when $n\geq 4$, we get the claimed failures there. When $d=2m+1$, then $H(n,2,1,j,m)$ is either $-2$ or $-4$. In any case, it is always negative, independent on $n$ and $j$, so it has maximal rank in this case. Finally, for $n=3$, we have that $H(3,2,2,j,m)=2j-5$ or $H(3,2,2,j,m)=2j-7$, which is non-negative only in the first case when $j=3$ and never in the second case. The first case happens when $3m-(3-3)$ is even, that is, when $m$ is even. Hence the only failure of $p_{2,3}$ to have maximal rank occurs when $d=2m+2$ for $m$ even, i.e. when $d\equiv 2 \pmod{4}$.
\end{proof}

\begin{lemma}
\label{lem:rest_failing}
For the following values of $n$ and $k$, multiplication by the form $p_{n,k}$ fails expected injectivity on $\mathbf{k}[x_1,\dots, x_n]/(x_1^d,\dots, x_n^d)$ exactly for the below values of $d$.
\begin{itemize}
\item $n=4, k=3$ and any $d$.
\item $n=3, k=5$ and any $d$.
\item $n=3, k=4$ and $d\not \equiv 2 \pmod{4}$.
\item $n=3, k=3$ and $d\not \equiv \pm 1 \pmod{6}$.
\end{itemize}
\end{lemma}

\begin{proof}
Since the proof is similar to Lemma \ref{lem:n=2} and Lemma \ref{lem:k=2}, we collect all necessary calculations in Figure \ref{fig:failing_table}.

\begin{figure}
\begin{center}
\rowcolors{2}{gray!15}{white}
\begin{tabular}{c|c|c|c|c|c}
$n$ & $k$ & $j$ & $m \pmod{2}$ & $H(n,k,i,j,m)$ & $d$ failing \\
\hline
$4$& $3$ & $0$ & $0$ & $5-4i$ & $3m+1$ \\
$4$& $3$ & $0$ & $1$ & $5-4i$ & $3m+1$ \\
$4$& $3$ & $4$ & $0$ & $4i-7$ & $3m+2, 3m+3$ \\
$4$& $3$ & $4$ & $1$ & $4i-7$ & $3m+2, 3m+3$ \\
$3$& $5$ & $0$ & $0$ & $2-3i$ & $-$ \\
$3$& $5$ & $0$ & $1$ & $7-3i$ & $5m+1, 5m+2$ \\
$3$& $5$ & $1$ & $0$ & $2-i$ & $5m+1, 5m+2$ \\
$3$& $5$ & $1$ & $1$ & $-3-i$ & $-$ \\
$3$& $5$ & $2$ & $0$ & $i-8$ & $-$  \\
$3$& $5$ & $2$ & $1$ & $i-3$ & $5m+3, 5m+4, 5m+5$ \\
$3$& $5$ & $3$ & $0$ & $3i-8$ & $5m+3, 5m+4, 5m+5$ \\
$3$& $5$ & $3$ & $1$ & $3i-13$ & $5m+5$ \\
$3$& $4$ & $0$ & $0$ & $1-3i$ & $-$ \\
$3$& $4$ & $0$ & $1$ & $5-3i$ & $4m+1$ \\
$3$& $4$ & $1$ & $0$ & $1-i$ & $4m+1$ \\
$3$& $4$ & $1$ & $1$ & $-3-i$ & $-$ \\
$3$& $4$ & $2$ & $0$ & $i-7$ & $-$ \\
$3$& $4$ & $2$ & $1$ & $i-3$ & $4m+3, 4m+4$ \\
$3$& $4$ & $3$ & $0$ & $3i-7$ & $4m+3, 4m+4$ \\
$3$& $4$ & $3$ & $1$ & $3i-11$ & $4m+4$ \\
$3$& $3$ & $0$ & $0$ & $-3i$ & $-$ \\
$3$& $3$ & $0$ & $1$ & $3-3i$ & $3m+1$ \\
$3$& $3$ & $1$ & $0$ & $-i$ & $-$ \\
$3$& $3$ & $1$ & $1$ & $-3-i$ & $-$ \\
$3$& $3$ & $2$ & $0$ & $i-6$ & $-$ \\
$3$& $3$ & $2$ & $1$ & $i-3$ & $3m+3$ \\
$3$& $3$ & $3$ & $0$ & $3i-6$ & $3m+2, 3m+3$ \\
$3$& $3$ & $3$ & $1$ & $3i-9$ & $3m+3$ \\
\end{tabular}
\end{center}
\caption{Failures of $p_{n,k}$ to be a max-rank element on $A$ given $d=km+i$ for some values of $n$ and $k$.}
\label{fig:failing_table}
\end{figure}

As an example, when $n=k=3$, the table can be read as follows. 
Recall from Proposition \ref{prop:Inj_char} that we get failure of injectivity when $d=km+i$ exactly when $H(3,3,i,j,m)\geq 0$ for some $j$. Looking at the rows where $n=k=3$, we then see that when $m$ is even, we get failure when $i=3m+2$ and $d=3m+3$, all coming from the row where $j=3$, but not when $d=3m+1$. Hence it has maximal rank when $d\equiv 1 \pmod{6}$ but not when $d\equiv 2,3 \pmod{6}$. When $m$ is odd, the table says that we have failures for $d=3m+1$ and $d=3m+3$, the first coming from $j=0$ and the other from $j=2$ or $j=3$, but not when $d=3m+2$. Hence it has maximal rank when $d\equiv 5 \pmod{6}$ but not when $d\equiv 4,6 \pmod{6}$, as claimed. The remaining results can be obtained from Figure \ref{fig:failing_table} in the same way. 
\end{proof}

We are now ready to state exactly when the power sum symmetric polynomial $p_{n,k}=x_1^k + \cdots + x_n^k$ is a max-rank element on $A=\mathbf{k}[x_1,\dots, x_n]/(x_1^d, \dots, x_n^d)$.

\begin{theorem}
\label{thm:most_failures}
Let $n\geq 2$. Then, multiplication by the power sum symmetric form $p_{n,k}=x_1^k+\cdots + x_n^k$ is a max-rank element on $A=\mathbf{k}[x_1,\dots, x_n]/(x_1^d, \dots, x_n^d)$, for some integer $d>k$, exactly for the values of $d$ indicated in Figure \ref{fig:Char_power_sum}.

\begin{figure}[h]
\begin{center}
\begin{tabular}{c|c|c|c}
\diagbox{$k$}{$n$} & $2$ & $3$ & $\geq 4$  \\
\hline
$2$ & $d$ odd & $d\not \equiv 2 \pmod{4}$ & $d$ odd \\
$3$ & $d\not\equiv 0 \pmod{3}$ & $d\equiv \pm 1 \pmod{6}$ & -\\
$4$ & $d\equiv 2 \pmod{4}$ & $d\equiv 2 \pmod{4}$ & -\\
$2m > 4$ & $d\equiv m \pmod k$ & - & -\\
$2m+1>4$ & $d\equiv m,m+1 \pmod k$ & - & -\\
\end{tabular}
\end{center}
\caption{All $d>k$ where $p_{n,k}$ is a max-rank element on $A$.}
\label{fig:Char_power_sum}
\end{figure}

\end{theorem}
\begin{proof}
We know that 
the statement 
is correct when $n=2$ or $k=2$ from Lemma~\ref{lem:n=2} and Lemma \ref{lem:k=2}. Taking the sporadic cases in Lemma \ref{lem:rest_failing} into account, we are left to show that $p_{n,k}$ is not a max-rank element for any $d>k$ in the following cases. When $k\geq 6$ if $n\geq 3$, when $k\geq 4$ if $n\geq 4$, and when $k\geq 3$ if $n\geq 6$. 

Write $d=km+i$ for some $1\leq i\leq k$. By Proposition \ref{prop:Inj_char}, we need to find $j$ such that $H(n,k,i,j,m)\geq 0$ where $H$ is as defined in \eqref{eq:H-function}. Let $H_1(n,k,i,j)$ be defined such that $H_1(n,k,i,j)=H(n,k,i,j,m)$ when $mn-(n-j)$ is even. Since $H(n,k,i,j,m)=H_1(n,k,i,j)$ for at least one of the cases when $j=0$ or $j=1$, as well as one of $j=n-1$ or $j=n$, it suffices to show that $H_1(n,k,i,j)\geq 0$ for both $j=0$ and $j=1$, or for both $j=n-1$ and $j=n$, to get our desired result. Noting that
\[
H_1(n,k,i,j) = j(2i-k) + nk-ni-n-k,
\]  
we see that $H_1$ is maximized when $j\in \{0,1\}$ if $2i\leq k$, and maximized when $j\in \{n-1,n\}$ if $2i\geq k$. 

Assume first that $2i\leq k$. Then 
\[
H_1(n,k,i,0) = nk-ni-n-k \geq nk-\frac{nk}{2}-n-k = \frac{nk}{2}-n-k
\]
and
\[
H_1(n,k,i,1) = 2i-k+nk-ni-n-k=(n-2)k-(n-2)i-n\geq \frac{(n-2)k}{2}-n.
\]
Note here that both of these lower bounds are equal, so let us call that function $h(n,k)$. Before we examine $h(n,k)$, let us look at the case when $2i\geq k$. Then
\[
H_1(n,k,i,n-1) = (n-2)i-n \geq \frac{(n-2)k}{2}-n = h(n,k)
\]
and
\[
H_1(n,k,i,n) = ni-n-k\geq \frac{nk}{2}-n-k = h(n,k).
\]
Hence, in all cases it suffices to show that $h(n,k)\geq 0$ to get the desired non-negativity for $H_1$. This is now a case of simple calculus. Taking the derivative with respect to $n$, we see that $h_n(n,k)=\frac{k}{2}-1\geq 0$ since $k\geq 2$ as always. Hence $h$ is increasing in $n$. The calculations $h(6,k)=2k-6\geq 0$ for $k\geq 3$ then shows that $h(n,k)\geq 0$ when $n\geq 6$ and $k\geq 3$, $h(4,k)=k-4\geq 0$ for $k\geq 4$ shows that $h(n,k)\geq 0$ when $n\geq 4$ and $k\geq 4$, and finally $h(3,k)=\frac{k}{2}-3\geq 0$ for $k\geq 6$ shows that $h(n,k)\geq 0$ when $n\geq 3$ and $k\geq 6$.
\end{proof}

\section{Schur polynomials}

In this section, we initiate the study of when Schur polynomials are max-rank elements on monomial complete intersections. Let us begin with their definition.

\begin{definition}
Let $\lambda=(\lambda_1,\dots, \lambda_n)$ be a partition. The \emph{Schur polynomial} $s_\lambda(x_1,\dots, x_n)$ is the symmetric polynomial defined by
\[
s_\lambda(x_1,\dots, x_n) = \frac{\det \begin{pmatrix}
x_1^{\lambda_1 + n-1} & x_1^{\lambda_2 + n-2} & \cdots & x_1^{\lambda_n}\\
x_2^{\lambda_1 + n-1} & x_2^{\lambda_2 + n-2} & \cdots & x_2^{\lambda_n}\\
\vdots & \vdots & \ddots & \vdots \\
x_n^{\lambda_1 + n-1} & x_n^{\lambda_2 + n-2} & \cdots & x_n^{\lambda_n}\\
\end{pmatrix}}{\det \begin{pmatrix}
x_1^{n-1} & x_1^{n-2} & \cdots & 1\\
x_2^{n-1} & x_2^{n-2} & \cdots & 1\\
\vdots & \vdots & \ddots & \vdots \\
x_n^{n-1} & x_n^{n-2} & \cdots & 1\\
\end{pmatrix}}.
\]
\end{definition}

The above is only one of many equivalent definitions of Schur polynomials. For others and an introduction to their rich theory, we recommend \cite[Chapter 7]{Stanley_Enum_Com_Vol_2}. The Schur polynomials also contains several interesting subfamilies of symmetric polynomials. For example, if $\lambda=(\lambda_1,0,\dots, 0)$ only has one part, then 
\[
s_{(\lambda_1,0,\dots, 0)} = h_{\lambda_1} = \sum_{a_1+\cdots + a_n=\lambda_1}x_1^{a_1}\cdots x_n^{a_n}
\]
is the complete symmetric homogeneous polynomial of degree $\lambda_1$. Another subfamily of interest appears when $\lambda=(1,1,\dots, 1,0,\dots, 0)$ only has parts of size one. If the number of non-zero parts is $k$, then
\[
s_{(1,1,\dots, 1,0,\dots, 0)} = e_k= \sum_{1\leq i_1< i_2<\cdots <i_k\leq n}x_{i_1}\cdots x_{i_n}
\]
is the $k$:th elementary symmetric polynomial in $n$ variables. 

For a general partition $\lambda$, it is a difficult task to determine for which $n$ and $d$ the Schur polynomial $s_\lambda(x_1,\dots, x_n)$ is a max-rank element on the monomial complete intersection generated in degree $d$. So let us begin by focusing on the case with few number of variables. 

\subsection{The case $n=2$}

Our goal now is, for any partition $(a,b)$, to completely determine when $s_{(a,b)}(x,y)$ is a max-rank element on the algebra $A=\mathbf{k}[x,y]/(x^d,y^d)$.

For our first observation, we have that
\[
s_{(a,b)}(x,y) = \frac{\begin{pmatrix}
x^{a+1} & x^b\\
y^{a+1} & y^b
\end{pmatrix}}{\begin{pmatrix}
x & 1\\
y & 1
\end{pmatrix}} = (xy)^b \frac{\begin{pmatrix}
x^{a-b+1} & 1\\
y^{a-b+1} & 1
\end{pmatrix}}{\begin{pmatrix}
x & 1\\
y & 1
\end{pmatrix}} =(xy)^bh_{a-b}(x,y).
\]
Since $h_{a-b}(x,y) = x^{a-b} + x^{a-b-1}y + \cdots + y^{a-b}$, we have a good understanding for how $s_{(a,b)}(x,y)$ looks like. Further, as $s_{(a,b)}$ has degree $a+b$, if $a+b\geq 2d-2$, then $s_{(a,b)}$ is trivially a max-rank element. Hence we may assume that $a+b<2d-2$ unless specified otherwise. 

Assume first now that $a+b=2m$. Then, by Lemma \ref{lem:inj_sufficient}, $s_{(a,b)}$ is a max-rank element if and only if multiplication by $s_{(a,b)}$ is bijective from degree $d-1-m$ to $d-1+m$. Now, $A_{d-1-m}$ has a basis given by $\{x^{d-1-m}, x^{d-m}y,\dots, y^{d-1-m}\}$ and $A_{d-1+m}$ a basis given by $\{x^{d-1}y^m, x^{d-2}y^{m+1},\dots, x^my^{d-1}\}$. Since $s_{(a,b)}$ is the sum of all monomials of degree $a+b$ that is divisible by $(xy)^b$, we see that the $(d-m)\times(d-m)$ matrix $M$ representing multiplication by $s_{(a,b)}$ with respect to these bases are given by
\[
M_{i,j} = 
\begin{cases}
1 & \text{ if } x^{d-m-i+b}y^{i-1+b} \text{ divides } x^{d-j}y^{m+j-1}, \\
0 & \text{ otherwise}.
\end{cases}
\]
Note that $x^{d-m-i+b}y^{i-1+b}$ divides $x^{d-j}y^{m+j-1}$ if and only if $x^{d-m-(i+1)+b}y^{i+b}$ divides $x^{d-(j+1)}y^{m+j}$. Hence every diagonal of $M$ is constant, i.e. $M$ is a Toeplitz matrix. Moreover, if $a\geq d-1$, then $M$ is the all ones matrix and does not have maximal rank when $a+b<2d-1$. Hence we can assume that $a<d-1$ from now on. Then, $M$ is the Toeplitz matrix defined by having the first row consisting of $a-m+1$ ones followed by zeroes and the first column having $m-b+1$ ones followed by zeroes.

\begin{example}
Let $d=7$ and $(a,b)=(3,1)$. Then the matrix representing multiplication by $s_{(3,1)}=(xy)h_2(x,y)$ on $A=\mathbf{k}[x,y]/(x^7,y^7)$ from degree $4$ to degree $8$ is
\[ M(3,3,5)=
\begin{pmatrix}
1 & 1 & 0 & 0 & 0\\
1 & 1 & 1 & 0 & 0\\
0 & 1 & 1 & 1 & 0\\
0 & 0 & 1 & 1 & 1\\
0 & 0 & 0 & 1 & 1\\
\end{pmatrix}.
\]
This matrix has determinant $0$, so $s_{(3,1)}$ is not a max-rank element on $A$.
\end{example}

Determinant formulas for different families of $(0,1)$ Toeplitz matrices have been studied recently, see e.g. \cite{Shitov}. Our main tool for determining when $s_{(a,b)}$ is a max-rank element is now a complete description for the determinant of Toeplitz matrices of the above form that may therefore be of independent interest.

\begin{theorem}\label{thm:Toeplitz_det_formula}
Let $M(r,c,n)$ be the Toeplitz matrix of size $n\times n$ with the first row given by $c$ ones followed by zeroes and with the first column given by $r$ ones followed by zeroes. Then
\[
\det(M(r,c,n)) = 
\begin{cases}
(-1)^{n(r-1)} & \text{ if } n\equiv 0 \pmod{r+c-1}\\
(-1)^{(n+1)(r-1)} & \text{ if } n\equiv 1 \pmod{r+c-1}\\
0 & \text{ otherwise}.
\end{cases}
\]
\end{theorem}

\begin{proof}
The key idea for the proof is to view $M(r,c,n)$ as a specialization of another matrix where a formula for the determinant is known. To this end, let $M_e(r,c,n)$ be the Toeplitz matrix with first row given by $e_{r-1}, e_r, \dots, e_{r+c-2}, 0, 0, \dots, $ and first column given by $e_{r-1}, e_{r-2},\dots, e_1, 1, 0, \dots$, where $e_i$ is the elementary symmetric polynomial of degree $i$ in the variables $x_1,\dots, x_{r+c-2}$. The determinant of $M_e(r,c,n)$ can then be calculated from the Jacobi-Trudi identity 
\[
s_{\lambda} = \det
\begin{pmatrix}
e_{\lambda_1'} & e_{\lambda_1'+1} & \cdots & e_{\lambda_1'+l-1} \\
e_{\lambda_2'-1} & e_{\lambda_2'} & \cdots & e_{\lambda_2'+l-2} \\
\vdots & \vdots & \ddots & \vdots \\
e_{\lambda_l'-l+1} & e_{\lambda_l'-l+2} & \cdots & e_{\lambda_l'}
\end{pmatrix}
\]
where $\lambda'$ is the conjugate partition to the partition $\lambda$. To get $M_e(r,c,n)$, we let $\lambda_i'=r-1$ for $i=1,2,\dots, n$ with $l=n$ and see that $$\det(M_e(r,c,n))=s_{(r-1)^n}(x_1,\dots, x_{r+c-2})$$
where the partition $\lambda=(r-1)^n$ is given by $\lambda_i=n$ for $i=1,\dots, r-1$, i.e. $(r-1)^n$ is a rectangle with $r-1$ rows and $n$ columns.

To connect $M_e(r,c,n)$ to $M(r,c,n)$, we now need to find a specialization of $x_1,\dots, x_{r+c-2}$ such that $e_i(x_1,\dots, x_{r+c-2})=1$ for all $i=1,\dots, r+c-2$. Recall the identity
\begin{equation*}
(t+x_1)\cdots (t+x_{r+c-2})= t^{r+c-2} + e_1(x_1,\dots, x_{r+c-2})t^{r+c-3} + \cdots + e_{r+c-2}(x_1,\dots, x_{r+c-2}).
\end{equation*}
From this, we see that we are looking for the roots of
\[
t^{r+c-2} + t^{r+c-3} + \cdots + t + 1 = \frac{t^{r+c-1}-1}{t-1}.
\]
But these are given by the $(r+c-1)$ roots of unity $q^j$ where $q=e^{\frac{2\pi i}{r+c-1}}$ and $j=1,2,\dots, r+c-2$. Hence
\begin{align*}
\det(M(r,c,n)) &= \left. \det(M_e(r,c,n))\right|_{x_j=-q^j} \\
&= s_{(r-1)^n}(-q, -q^2,\dots, -q^{r+c-2}) \\
&=(-q)^{n(r-1)}\cdot  s_{(r-1)^n}(1,q,\dots, q^{r+c-3}),
\end{align*}
where we used that a Schur polynomial $s_\lambda$ is homogeneous of degree $|\lambda|$. 

Our task now is to evaluate $s_{(r-1)^n}(1,q,\dots, q^{r+c-3})$. Luckily, this is known as a principal evaluation of the Schur function and Stanley's hook formula \cite[Theorem~7.21.2]{Stanley_Enum_Com_Vol_2} gives that
\[
s_{(r-1)^n}(1,q,\dots, q^{r+c-3}) = q^{b((r-1)^n)}\prod_{u\in (r-1)^n} \frac{1-q^{r+c-2+c(u)}}{1-q^{h(u)}}
\]
where the product goes over all entries $u=(i,j)$ of the partition $(r-1)^n$, 
\begin{equation*}
b((r-1)^n) = \sum_{k=1}^{r-1}(k-1)n = \frac{n(r-1)(r-2)}{2},
\end{equation*}
for $u=(i,j)$, $c(u)=j-i$ is the content at $u$, and $h(u)=(n-j)+(r-1-i)+1$ is the hook length at $u$. Therefore
\[
s_{(r-1)^n}(1,q,\dots, q^{r+c-3}) = q^{\frac{n(r-1)(r-2)}{2}}\prod_{i=1}^{r-1}\prod_{j=1}^n \frac{1-q^{r+c-2+j-i}}{1-q^{(n-j)+(r-1-i)+1}}.
\]
After re-indexing, our current goal is then to calculate
\[
\prod_{i=0}^{r-2}\prod_{j=0}^{n-1} \frac{1-q^{c+n-1+i-j}}{1-q^{j+i+1}}.
\]
Recalling the definition of the $q$-binomial coefficient
\[
\binom{a}{b}_q = \frac{(1-q^a)(1-q^{a-1})\cdots (1-q^{a-b+1})}{(1-q)\cdots (1-q^b)},
\]
a short calculation gives 
\begin{equation}\label{eq:q-binom_prod}
\prod_{i=0}^{r-2}\prod_{j=0}^{n-1} \frac{1-q^{c+n-1+i-j}}{1-q^{j+i+1}} = \prod_{i=0}^{r-2} \frac{\binom{c+n-1+i}{n+i}_q}{\binom{c-1+i}{i}_q}.
\end{equation}
Here we are again lucky that there are ways to simplify evaluations of $q$-binomial coefficients when $q$ is a root of unity. If $q$ is an $e$:th root of unity, then the $q$-Lucas theorem \cite{q-Lucas} says that if $a=ke+s$ and $b=le+t$ with $0\leq s,t<e$, then 
\[
\binom{a}{b}_q = \binom{k}{l}\cdot \binom{s}{t}_q.
\]
Setting $e=r+c-1$ and writing $n+i=le+t$, we get that
$c-1+n+i=le+(t+c-1)$ if $t<r$ and $c-1+n+i=(l+1)e+(t-r)$ if $t\geq r$. Hence
\[
\binom{c+n-1+i}{n+i}_q = \begin{cases}
\binom{c-1+t}{t}_q & \text{ if } t<r \\
\binom{l+1}{l}\cdot \binom{t-r}{t}_q = 0 & \text{ otherwise}.
\end{cases}
\]
So the only way for $\binom{c+n-1+i}{n+i}_q$ to be non-zero for all $i=0,\dots, r-2$, so that the product in \eqref{eq:q-binom_prod} is non-zero, is that $n\equiv 0 \pmod{r+c-1}$ or $n\equiv 1 \pmod{r+c-1}$, in all other cases, the whole sought determinant is zero.

When $n\equiv 0 \pmod{r+c-1}$, $n=el$ for some $l$ and $\binom{c+n-1+i}{n+i}_q=\binom{c-1+i}{i}_q$, so the product in \eqref{eq:q-binom_prod} simplifies to $1$. If instead $n\equiv 1 \pmod{r+c-1}$, then nearly everything in \eqref{eq:q-binom_prod} directly simplifies to $1$, remembering that $\binom{c-1}{0}_q=1$, we are only left to evaluate $\binom{c-1+r-1}{r-1}_q$. From the definition of the $q$-binomial coefficient we have that
\[
\binom{c-1+r-1}{r-1}_q = \prod_{i=1}^{r-1}\frac{1-q^{c+r-1-i}}{1-q^i}.
\]
We now note that since $q$ is the principal $(r+c-1)$ root of unity, $q^{r+c-1-i}=q^{-i}$, giving that
\[
\frac{1-q^i}{1-q^{c+r-1-i}} = \frac{1-q^i}{1-q^{-i}} = \frac{q^i(1-q^i)}{q^i-1} = -q^i.
\]
Therefore
\[
\prod_{i=1}^{r-1}\frac{1-q^{c+r-1-i}}{1-q^i} = \prod_{i=1}^{r-1} (-q)^{-i} = (-1)^{r-1}q^{-\frac{r(r-1)}{2}}.
\]
Putting everything together, when $n\equiv 0 \pmod{r+c-1}$ we have that $$\det(M(r,c,n))=(-q)^{n(r-1)}\cdot q^{\frac{n(r-1)(r-2)}{2}}=(-1)^{n(r-1)}$$
by using that $q^n=1$ in this case. Next, when $n\equiv 1 \pmod{r+c-1}$, then $q^n=q$ and we get that
\begin{align*}
\det(M(r,c,n))&= (-q)^{n(r-1)}\cdot q^{\frac{n(r-1)(r-2)}{2}}\cdot (-1)^{r-1}\cdot q^{-\frac{r(r-1)}{2}} \\
&=(-1)^{(n+1)(r-1)}\cdot q^{n(r-1)+\frac{n(r-1)(r-2)}{2}-\frac{r(r-1)}{2}} \\
&=(-1)^{(n+1)(r-1)}\cdot q^{r-1+\frac{(r-1)(r-2)}{2}-\frac{r(r-1)}{2}}\\
&=(-1)^{(n+1)(r-1)}\cdot q^0 = (-1)^{(n+1)(r-1)}.
\end{align*}
Remembering that $\det(M(r,c,n))=0$ in all remaining cases, we are done.
\end{proof} 

For the case when $a+b$ is even, we can now say when $s_{(a,b)}$ is a max-rank element.

\begin{lemma}\label{lem:even_complete_homog}
The Schur polynomial $s_{(a,b)}$, where $a+b=2m$, is a max-rank element on $\mathbf{k}[x,y]/(x^d,y^d)$ if and only if $m\geq d-1$, $d\equiv m \pmod{a-b+1}$ or $d\equiv m+1 \pmod{a-b+1}$.
\end{lemma}
\begin{proof}
If $m\geq d-1$, then $s_{(a,b)}$ has large enough degree that it trivially is a max-rank element. Else, if $m<d-1$ and $a<d-1$, then the discussion at the start of this sections gives that the multiplication by $s_{(a,b)}$ from degree $d-1-m$ to $d-1+m$ is represented by $M(a-m+1,m-b+1,d-m)$ and $s_{(a,b)}$ is a max-rank element if and only if $\det(M(a-m+1,m-b+1,d-m))\neq 0$. By Theorem \ref{thm:Toeplitz_det_formula}, $\det(M(a-m+1,m-b+1,d-m))$ is non-zero if and only if
\[
d-m \equiv 0,1 \pmod{a-b+1} \iff d\equiv m, m+1 \pmod{a-b+1}
\]
as claimed. Finally, note that if $m<d-1$ and $a\geq d-1$, then the multiplication by $s_{(a,b)}$ from degree $d-1-m$ to $d-1+m$ is represented by the all ones matrix and fails to have maximal rank. This is in accordance with our conditions since no $d$ in this interval satisfies the modulo condition. Indeed, the chain of inequalities
\[
m + (a-b+1) = 2a-m+1 \geq 2(d-1) -m +1 > 2(d-1) -(d-1) + 1=d>m+1,
\]
obtained by using our assumptions, gives that $m$, $m+1$ and $d$ all lie in different equivalence classes $\pmod{a-b+1}$.
\end{proof}

Let us now turn to the case when $a+b=2m+1$. For the failures of the max rank property, we will in this case construct an explicit kernel element.

\begin{lemma}\label{lem:Kernel_complete_homog}
If $a+b=2m+1$, $a+b<2d-2$ and $d$ is not equal to $m,m+1$ or $m+2 \pmod{a-b+1}$, then $s_{(a,b)}$ is not a max-rank element on $\mathbf{k}[x,y]/(x^d,y^d)$.
\end{lemma}

\begin{proof}
Assume that $2m+1<2d-2$ so we do not have maximal rank for trivial reasons. The key relation for us is then
\begin{equation}\label{eq:Kernel_complete_homog}
(x-y)\cdot h_{l-1}(x^{a-b+1},y^{a-b+1})\cdot s_{(a,b)}-y^bx^{(a-b+1)l+b}+x^by^{(a-b+1)l+b} = 0.
\end{equation}
To prove \eqref{eq:Kernel_complete_homog}, we can first clear a common factor $(xy)^b$ since $s_{(a,b)}(x,y)=(xy)^bh_{a-b}(x,y)$. Recalling that $h_{k-1}(x,y)=\frac{x^{k}-y^k}{x-y}$, we have
\begin{align*}
&(x-y)\cdot h_{l-1}(x^{a-b+1},y^{a-b+1})\cdot h_{a-b}-x^{(a-b+1)l}+y^{(a-b+1)l}\\
=&(x-y)\cdot \frac{x^{(a-b+1)l}-y^{(a-b+1)l}}{x^{a-b+1}-y^{a-b+1}} \cdot \frac{x^{a-b+1}-y^{a-b+1}}{x-y}-x^{(a-b+1)l}+y^{(a-b+1)l} \\
=&0.
\end{align*}
When $d\leq (a-b+1)l +b$, \eqref{eq:Kernel_complete_homog} exhibits $(x-y)h_{l-1}(x^{a-b+1}, y^{a-b+1})$ as a non-zero kernel element of multiplication by $s_{a,b}$. To see when this is unexpected, write $d=(a-b+1)l+b-q$ for some natural number $q$. Since the kernel element has degree $1+(l-1)(a-b+1)$, we know that injectivity is expected in this degree if 
\[
2(1+(l-1)(a-b+1)) + a + b \leq 2d-2,
\]
which simplifies to 
\begin{equation}\label{eq:q_ineq}
2q+3\leq a-b+1
\end{equation}
Note that this element is indeed non-zero since it has a leading term $x^{(a-b+1)(l-1)+1}$, which a short calculation shows has an exponent strictly less than $d$ as long as $q<a$, which is satisfied if \eqref{eq:q_ineq} holds.
Before examining \eqref{eq:q_ineq} further, by multiplying \eqref{eq:Kernel_complete_homog} by $(xy)^q$, we get a non-zero kernel element of degree $1+(l-1)(a-b+1) + 2q$ when $d=(a-b+1)l + b + q$. Now, this is an unexpected kernel element if
\[
2(1+(l-1)(a-b+1) + 2q) + a + b \leq 2d-2,
\]
which also simplifies to 
\[
2q+3\leq a-b+1.
\]
Again, a short calculation shows that it always has a non-zero leading term. Hence we have an unexpected kernel element when $d=(a-b+1)l + b + q$ and $2|q| + 3 \leq a-b+1$. Since $a-b+1=2(m-b+1)$, we see that $|q|$ can take the values $0,1,\dots, m-b-1$. In other words, $s_{(a,b)}$ is not a max-rank element if the value of $d \pmod{a-b+1}$ lies in the set $\{b-(m-b-1), b-(m-b-2),\dots, b, b+1, \dots, b+(m-b-1)\}$. Since $$b-(m-b-1)\equiv 2b-m+1 + a -b + 1 = a+b-m+2 = m+3 \pmod{a-b+1},$$
this set is the same as $\{m+3, m+4,\dots, b, b+1, \dots, m-1\}$, and therefore covers all cases except for when $d\equiv m, m+1, m+2 \pmod{a-b+1}$ as desired.
\end{proof}

With this, we can now state exactly when $s_{(a,b)}$ is a max-rank element on $\mathbf{k}[x,y]/(x^d,y^d)$.

\begin{theorem}\label{thm:Schur_n=2}
The Schur polynomial $s_{(a,b)}$ is a max-rank element on the algebra $\mathbf{k}[x,y]/(x^d,y^d)$ if and only if either
\begin{itemize}
\item $a+b\geq 2d-2$, or
\item $d\equiv m, m+1 \pmod{a-b+1}$ if $a+b=2m$, or
\item $d\equiv m, m+1, m+2 \pmod{a-b+1}$ if $a+b=2m+1$.
\end{itemize}
\end{theorem}

\begin{proof}
If $a+b\geq 2d-2$, we know that $s_{(a,b)}$ is a max-rank element trivially, so assume that $a+b<2d-2$. If $a+b=2m$, then the statement follows from Lemma~\ref{lem:even_complete_homog}. If $a+b=2m+1$, then we know that multiplication by $s_{(a,b)}$ is a max-rank element if and only if it is injective from degree $d-1-m-1$ to degree $d-1+m$. Now, assume that $a<d$. Then this is represented by a $(d-m-1)\times (d-m)$ Toeplitz matrix $M$ where the first row contains $a-m+1$ ones followed by zeroes, and the first column contains $m-b+1$ ones followed by zeroes. Hence, by removing the last column we get the matrix $M(m-b+1,a-m+1,d-m-1)$, which by Theorem~\ref{thm:Toeplitz_det_formula} has maximal rank when $d\equiv m+1, m+2 \pmod{a-b+1}$. Therefore $M$ also has maximal rank in those cases. If we instead add a row at the bottom of $M$ to obtain $M(m-b+1, a-m+1,d-m)$, then this has maximal rank when $d\equiv m,m+1 \pmod{a-b+1}$, and thus $M$ has maximal rank in those cases as well. Moreover, in the remaining cases, Lemma \ref{lem:Kernel_complete_homog} gives that $s_{(a,b)}$ is not a max-rank element. Finally, if $a\geq d$, then multiplication by $s_{(a,b)}$ is represented by the all ones matrix of size $(d-m-1)\times (d-m)$. This does not have maximal rank unless $d=m+2$ where we note that $d\geq m+2$ is a requirement to satisfy $a+b<2d-2$. That $d=m+2$ is also the only case covered by our criterion since if $d\geq m+3$, we have that
\[
m + (a-b+1) = 2a - m \geq 2d - (d-3) > d,
\]
showing that $d$ can not satisfy the congruence in any other instance.
\end{proof}

\begin{corollary}\label{cor:e_and_h_for_n=2}
The elementary symmetric polynomials $e_1=x+y$ and $e_2=xy$ are both max-rank elements on $A=\mathbf{k}[x,y]/(x^d,y^d)$ for all $d\geq 2$. Moreover, the complete homogeneous symmetric polynomial $h_a$ is a max rank element on $A$ if and only if 
\begin{itemize}
\item $a\geq 2d-2$, or
\item $d\equiv m, m+1 \pmod{a+1}$ if $a=2m$, or
\item $d\equiv m, m+1, m+2 \pmod{a+1}$ if $a=2m+1$.
\end{itemize}
\end{corollary}

\begin{proof}
This follows directly from Theorem \ref{thm:Schur_n=2} since $e_1=s_{(1,0)}$, $e_2=s_{(1,1)}$ and $h_a=s_{(a,0)}$.
\end{proof}

\subsection{The case $n \geq 3$}

After obtaining the full characterization for when Schur polynomials define max-rank elements in two variables, the natural question is then what happens when we increase the number of variables. In that case, the behavior is still a mystery to us, but let us mention a partial result.

First, if $\lambda=(1,0,\dots, 0)$, then $s_\lambda = e_1 = h_1 = x_1+\cdots + x_n$ is known to be a max-rank element in $A=\mathbf{k}[x_1,\dots, x_n]/(x_1^d, \dots, x_n^d)$ since $A$ has the strong Lefschetz property. Further, note that a Schur polynomial $s_\lambda$ where $\lambda=(\lambda_1,\dots, \lambda_n)$ is irreducible if and only if $\lambda_n=0$ and $\gcd(\lambda_1+n-1, \lambda_2 + n-2, \dots, \lambda_n)=1$ by \cite[Theorem 3.1]{Irred_Schur}. If $s_\lambda$ is reducible, then the next result shows that it is not a max-rank element on $A$, at least if $d$ is not to small. 

\begin{proposition}\label{prop:Schur_reducible}
If $n\geq 3$ and $s_\lambda$ is reducible, then $s_\lambda$ is not a max-rank element on $\mathbf{k}[x_1,\dots, x_n]/(x_1^d, \dots, x_n^d)$ when $$d\geq \frac{n+|\lambda|}{n-2}.$$
\end{proposition}

\begin{proof}
Assume first that $\lambda_n\neq 0$. It then follows from the definition of Schur polynomials that 
\[
s_\lambda = (x_1\cdots x_n)^{\lambda_n}s_{(\lambda_1-\lambda_n, \lambda_2-\lambda_n,\cdots, 0)}.
\]
Since $\lambda_n\neq 0$, we therefore have that $x_1^{d-\lambda_n}$ is a non-zero kernel element of multiplication by $s_\lambda$ in $A=\mathbf{k}[x_1,\dots, x_n]/(x_1^d, \dots, x_n^d)$. This is an obstruction to being a max-rank element if
\[
2(d-\lambda_n) + |\lambda| \leq n(d-1),
\]
which simplifies to
\[
\frac{n + |\lambda| -2\lambda_n}{n-2}\leq d.
\]
Next, if $\lambda_n=0$, then it again follows from the definition of Schur functions that 
\[
V_g(x_1,\dots, x_n) = \prod_{1\leq i<j\leq n}\frac{x_i^g-x_j^g}{x_i-x_j},
\]
where $g=\gcd(\lambda_1+n-1, \lambda_2 + n-2, \dots, \lambda_n)$, is a factor of $s_\lambda$. Moreover, by \cite[Theorem 3.1]{Irred_Schur}, the polynomial $s_\lambda/V_g$ is either constant or irreducible. Hence, if $\lambda_n=0$ and $s_\lambda$ is reducible, then $g>1$ and $s_\lambda$ has the factor $V_g$. Assume $s_\lambda$ is on that form and write $d=gq+r$ where $0<r\leq g$. We then claim that the polynomial 
\[
f=(x_1-x_2)\cdot h_q(x_1^g, x_2^g)
\]
is a non-zero kernel element of multiplication by $s_\lambda$. It is non-zero since no cancellations occur in $f=(x_1-x_2)\cdot h_q(x_1^g, x_2^g)$ as $g\geq 2$, so for example $x_1x_2^{gq}$ is a non-zero term. To see that it lies in the kernel, note that $h_{g-1}(x_1,x_2)$ is a factor of $V_g$. The calculation 
\[
h_{g-1}(x_1,x_2)\cdot f = \frac{x_1^{g}-x_2^g}{x_1-x_2}\cdot \frac{x_1^{g(q+1)}-x_2^{g(q+1)}}{x_1^g-x_2^g}(x_1-x_2) = x_1^{g(q+1)}-x_2^{g(q+1)}
\]
together with $g(q+1)\geq d$, then shows that $f$ is a kernel element. To see when injectivity is expected in the degree where $f$ lives, we get the inequality
\[
2(gq+1) + |\lambda| \leq n(d-1),
\]
which simplifies to 
\[
\frac{n+|\lambda| + 2(1-r)}{n-2}\leq d.
\]
Since we assumed that $r>0$, it is therefore enough in both cases with 
\begin{equation*}
\frac{n+|\lambda|}{n-2}\leq d.    
\end{equation*}
\end{proof}

\section{The elementary symmetric polynomial}

Let us now turn our focus to the elementary symmetric polynomials 
\[
e_k = \sum_{1\leq i_1< i_2 <\cdots < i_k\leq n}x_{i_1}\cdots x_{i_k}.
\]
From Corollary \ref{cor:e_and_h_for_n=2}, we know exactly for which values of $k$ and $d$ that $e_k$ is a max-rank element on $\mathbf{k}[x,y]/(x^d, y^d)$. However, in three or more valuables, the classification problem is more difficult. Therefore, in the remainder of this section, we consider $A=\mathbf{k}[x_1,\dots, x_n]/(x_1^d,\dots, x_n^d)$ and $e_d$ of the same degree $d$ as the exponents defining the complete intersection. In this case, we can give a conjectural classification and prove that the conditions are necessary.

\begin{conjecture} \label{conj:e}
Let $n \geq 3$ and $d \leq n$. Then $e_d$ is a maximal rank element on $A$
if and only if either
\begin{itemize}
    \item $d=2,$ or
    \item $d=3$ and $n \not \equiv 2 \pmod{3}.$ 
\end{itemize}
\end{conjecture}

We now prove the necessity part of Conjecture \ref{conj:e}. 

\begin{proposition}\label{prop:e-fails-max-rank}
Let $n \geq 3$ and $1\leq d \leq n$. Then $e_d$ is not a maximal rank element on $A$ if $d\geq 4$, or $d=3$ and $n \equiv 2 \pmod{3}$.
\end{proposition}

\begin{proof}
The idea of the proof is to find an element $F$ in the inverse system of $(x_1^d,\dots, x_n^d, e_d)$ in high enough degree to show that $e_d$ fails expected surjectivity on $A$ in some degree. Let 
\[
V(n) = \sum_{\sigma\in S_n}\sgn(\sigma)\; x_{\sigma (1)}^{n-1}x_{\sigma (2)}^{n-2}\cdots x_{\sigma (n-1)} = \prod_{1\leq i<j\leq n}(x_i-x_j)
\]
be the Vandermonde polynomial in $n$ variables. We know that $e_d$ annihilates $V(n)$ by \cite[Example 2.28]{Lefschetz_book}. Next, let $e_i(a,a+1,\dots, b)$ denote the elementary symmetric polynomial of degree $i$ in the variables $x_a, x_{a+1},\dots, x_b$ and 
\[
V(a,\dots, b) = \prod_{a\leq i<j\leq b}(x_i-x_j)
\]
the Vandermonde polynomial in the variables $x_a,\dots, x_b$. Write $(s-1)d+1\leq n+1\leq sd$ for some $s\geq 2$. Then
\begin{equation}\label{eq:e_d-expansion}
e_d = \sum_{\substack{i_1+\cdots + i_s=d \\ i_j\geq 0}}e_{i_1}(1,2,\dots, d-1)e_{i_2}(d,d+1,\dots, 2d-1)\cdots e_{i_s}((s-1)d,\dots, n)
\end{equation}
and we claim that this annihilates the form 
\[
F=(x_1\cdots x_{d-1})^{d-1}V(d,\dots, 2d-1)\cdots V((s-1)d,\dots, n).
\]
Indeed, since $e_{i_j}((j-1)d,\dots, jd-1)$ annihilates $V((j-1)d,\dots, jd-1)$ for $j\geq 2$, it suffices that all non-zero terms in \eqref{eq:e_d-expansion} has some index $i_j$ with $j\geq 2$ non-zero. But if $i_j=0$ for all $j\geq 2$, then this forces $i_1=d$, giving the term $e_d(1,\dots, d-1)=0$. Hence $e_d$ annihilates $F$. Further, $x_i^d$ annihilates $F$ for all $i$. This is clear when $i\leq d-1$, while for higher $i$, it follows since each Vandermonde factor in $F$ contains at most $d$ variables and hence monomials with exponent at most $d-1$. Thus $F$ lies in the inverse system of $(x_1^d, \dots, x_n^d, e_d)$.

Next, to see when an element of degree $D$ in the inverse system of $(x_1^d,\dots, x_n^d, e_d)$ is unexpected, we recall that $e_d$ has expected surjectivity to degree $D$ if $D + (D-d) \geq n(d-1)$. Since $F$ has degree $(d-1)^2 + (s-2)\binom{d}{2} + \binom{n-(s-1)d+1}{2}$, the form $F$ is unexpected when
\[
2\left((d-1)^2 + (s-2)\binom{d}{2} + \binom{n-(s-1)d+1}{2}\right) -d \geq n(d-1),
\]
which can be simplified to
\[
g(n,d)=(sd-(n+1))^2+(1-s)d^2 + (n-2)d+1\geq 0.
\]
To see when $g(n,d)$ is non-negative, differentiating with respect to $n$ shows that $g$ has a minimum at $n=sd-d/2-1$. Thus, since
\[
g(sd-\frac{d}{2}-1,d)=\frac{3d^2}{4}-3d+1\geq 0
\]
if $d\geq 4$, we get that $g(n,d)\geq 0$ if $d\geq 4$ and that $e_d$ fails expected surjectivity for those values of $d$. Finally, if $d=3$, then $n\in \{3s-1, 3s-2, 3s-3\}$. Computing $g(n,3)$ at those values gives $g(3s-1,3)=1>0$ and $g(3s-2,3)=g(3s-3,3)=-1<0$. Hence $e_3$ is not a max-rank element if $n$ is of the form $n=3s-1$, or equivalently, if $n\equiv 2 \pmod{3}$.
\end{proof} 

\section{The complete homogeneous symmetric polynomial}

For our final family of symmetric polynomials, we now consider the complete homogeneous symmetric polynomials 
\[
h_k=\sum_{a_1+\cdots + a_n=k}x_1^{a_1}\cdots x_n^{a_n}.
\]
Similarly to the case of the elementary symmetric polynomials, we do have a full classification of when $h_k$ is a max-rank element on $\mathbf{k}[x,y]/(x^d, y^d)$, but in three or more variables, such a classification has not been found. Instead, for the remainder of this section, we consider $A=\mathbf{k}[x_1,\dots, x_n]/(x_1^d,\dots, x_n^d)$ and $h_d$ of the same degree $d$ as the exponents defining the complete intersection. In this case, we can give a conjectural classification and some partial results.

\begin{conjecture} \label{conj:h}
Let $n \geq 3$. Then $h_d$ is a maximal rank element on $A$ if and only if either
\begin{itemize}
    \item $d=2,$
    \item $d=3$ and $n \not \equiv 1 \pmod{3},$ or 
    \item $d=4$ and $n=3$. 
\end{itemize}
\end{conjecture}

We have not been able to prove Conjecture \ref{conj:h}, but for each fixed $n$, we can show that $h_d$ fails to be a max-rank element on $A$ for all but finitely many values of $d$. 

\begin{proposition}\label{prop:h_d_fails_max_rank}
The complete homogeneous polynomial $h_d$ fails to be a max-rank element on $A$ if $n\geq 3$ and $d\geq n+2$.
\end{proposition}

\begin{proof}
The idea here is the same as in the proof of Proposition \ref{prop:e-fails-max-rank}, we want to find a form $F$ of high enough degree in the inverse system of $(x_1^d, \dots, x_n^d, h_d)$ to show that $h_d$ fails expected surjectivity in some degree on $A$. By \cite[Example~2.87]{Lefschetz_book}, $h_d$ annihilates $(x_1\cdots x_n)^sV(n)$ if $d>s$ where $V(n)$ is the Vandermonde polynomial in $n$ variables. See also \cite[Proposition 2.1]{Bergeron_duals} or \cite{Satoru_Isogawa2025} for more details. Further, the largest $s$ such that $x_i^d$ annihilates $(x_1\cdots x_n)^sV(n)$ for all $i$ is $d-n$. So let $F=(x_1\cdots x_n)^{d-n}V(n)$. To see when an element of degree $D$ in the inverse system of $(x_1^d,\dots, x_n^d, h_d)$ is unexpected, we recall that $h_d$ has expected surjectivity to degree $D$ if $D + (D-d) \geq n(d-1)$. Since $F$ has degree $\binom{n}{2} + n(d-n)$, the form $F$ is unexpected when
\[
2\left(\binom{n}{2} + n(d-n)\right) -d \geq n(d-1),
\]
which can be simplified to
\[
d\geq n+1+\frac{1}{n-1}.
\]
Since $d$ and $n\geq 3$ are integers, we therefore get that $F$ is unexpected when $d\geq n+2$.
\end{proof}

Similarly to the proof of Proposition \ref{prop:e-fails-max-rank}, we believe that all failures of the max-rank property for $h_d$ in Conjecture \ref{conj:h} can be explained by unexpected elements in the inverse system of $(x_1^d,\dots, x_n^d, h_d)$ of the form
\[
\prod_{j=1}^{s}(x_{i_j}\cdot x_{i_j+1}\cdots x_{i_{j+1}-1})^{r_j}\cdot V(i_j,\dots,i_{j+1}-1)
\]
where $V(a,\dots, b)$ is the Vandermonde polynomial in the variables $x_a,\dots, x_b$ and $s,i_1,\dots, i_{s+1}, r_1,\dots, r_s$ are some suitably chosen numbers. 

\section{Outlook and further observations}

In this paper, we classified several instances of when specific symmetric polynomials are max-rank elements on an equigenerated monomial complete intersection. However, it is still an open question for most common symmetric polynomials to classify on which complete intersections they are max-rank elements. In this section, we now state some natural next questions in this direction and give some further conjectures based on computations done in Macaulay2 \cite{M2}. 

The main and most ambitious open question is to classify when Schur polynomials are max-rank elements on monomial complete intersections.

\begin{question}
Given a partition $\lambda$ with $k$ parts, find all integers $n\geq k$ and $d\geq 2$ such that the Schur polynomial $s_{\lambda}(x_1,\dots, x_n)$ is a max-rank element on the algebra $\mathbf{k}[x_1,\dots, x_n]/(x_1^d,\dots, x_n^d)$. 
\end{question}

It would be interesting to find other properties of the partition $\lambda$ that force success or failure of being a max-rank element for some values of $n$ and $d$ similar to Proposition \ref{prop:Schur_reducible}. Another possible way forward could be to classify the asymptotic behavior. For example, fixing a partition $\lambda$ and a number of variables $n$, can one say something about when $s_{\lambda}$ is a max-rank element on a complete intersection defined by $(x_1^d,\dots, x_n^d)$ when $d$ is large enough? We have the following conjecture in the three variable case.

\begin{conjecture}
Let $\lambda$ be a partition with at most three parts, $s_\lambda$ a Schur polynomial and $A=\mathbf{k}[x_1,x_2,x_3]/(x_1^d, x_2^d, x_3^d)$. Consider the two sets
\begin{align*}
F&=\{(\lambda_1, \lambda_2, \lambda_3) \; | \; \lambda_3\neq 0  \text{ or } \gcd(\lambda_1+2, \lambda_2+1)>1\}\cup \{(s,s,0) \; | \; s\geq 5\},\\
P&=\{(2,2,0), (3,3,0), (4,4,0)\}.
\end{align*}
Then the behavior of if $s_{\lambda}$ is a max-rank element or not on $A$ has the following description for all $d$ large enough.
\begin{itemize}
\item It fails if $\lambda\in F$.
\item It becomes periodic in $d$ if $\lambda\in P$, with the condition of being a max-rank element being 
\begin{itemize}
\item[$\star$] $d\equiv 2,4,5 \pmod{6}$ for $\lambda=(2,2,0)$,
\item[$\star$] $d\equiv 6 \pmod{8}$ for $\lambda=(3,3,0)$, and
\item[$\star$] $d\equiv 8 \pmod{10}$ for $\lambda = (4,4,0)$.
\end{itemize}
\item It is a max-rank element if $\lambda \notin F\cup P$.
\end{itemize}
\end{conjecture}

One can of course also study other properties of the ideals $(x_1^d,\dots, x_n^d, p)$ for some symmetric polynomial $p$. For example, following Conjecture \ref{conj:e} and Conjecture~\ref{conj:h}, the Hilbert series for the ideals
$(x_1^3,\ldots,x_n^3, h_3)$ and 
$(x_1^3,\ldots,x_n^3, e_3)$ should agree when $n$ is a multiple of $3$. We believe in fact that their Betti tables also agree.

\begin{conjecture}
The ideals  $(x_1^3,\ldots,x_n^3, h_3)$ and 
$(x_1^3,\ldots,x_n^3, e_3)$ have the same Betti numbers if and only if $n=3a$.
\end{conjecture}
Here it is not enough that $e_3$ and $h_3$ are symmetric polynomials with the max-rank property on the algebra $A=\mathbf{k}[x_1,\dots, x_n]/(x_1^3,\dots, x_n^3)$. We know that the polynomial $(x_1+\cdots + x_n)^3$ is also a max-rank element on $A$, but for $n = 6$, the Betti table for 
$(x_1^3,\ldots,x_6^3, (x_1+\cdots x_6)^3)$ differs from the others. A third Betti table from a max-rank element can also be found from
$(x_1^3,\ldots,x_6^3, f)$, where $f$ is generic, but not necessarily symmetric, of degree three.

Further, when a symmetric polynomial $p$ is not a max-rank element, the amount by which it fails do in many cases seem to have nice combinatorial interpretations. Let $A=\mathbf{k}[x_1,\dots, x_n]/(x_1^d,\dots, x_n^d)$ as usual. By Proposition \ref{prop:Failing_symmetric} and the fact that $A$ has a unimodal Hilbert series, the series given by 
$\HS(A/(p);t)-[(1-t^{\deg(p)}\HS(A;t))],$
where $[(1-t^{\deg(p)}\HS(A;t))]$ is the expected Hilbert series, is symmetric. Moreover, as $d$ increases, experiments suggests that the coefficients of this difference stabilizes. For example, when $n=5$, $e_4$ fails to be a max-rank element on $A$ for $d=5,7,9$ by the amounts $10t^{11}(1+5t+t^2), 10t^{15}(1+5t+14t^2+5t^3+t^4)$ and $10t^{19}(1 + 5t + 14t^2 + 30t^3 + 14t^4 + 5t^5 + t^6)$. The sequence of interest here thus starts $10\cdot(1,5,14,30,\dots,)$ which we guess is ten times the square pyramidal numbers $a_i$ where $a_i=\sum_{j=1}^i j^2$. 

Other interesting cases are for example when considering $e_5$ for $n=6$, in which case the difference has coefficients looking like $15\cdot\binom{i+3}{4}$, or $e_6$ when $n=7$, in which case the difference has coefficients giving the sequence $21\cdot(1,6,20,51,111,\dots,)$. To our surprise, the sequence $1,6,20,51,111,\dots,$ is the starting segment of a unique sequence in the OEIS \cite{OEIS_2018} given by A266760, Growth series for affine Coxeter group $D_5$, thus hinting at a possible connection to Coxeter groups.

\begin{figure}[h]
\includegraphics[scale=0.33]{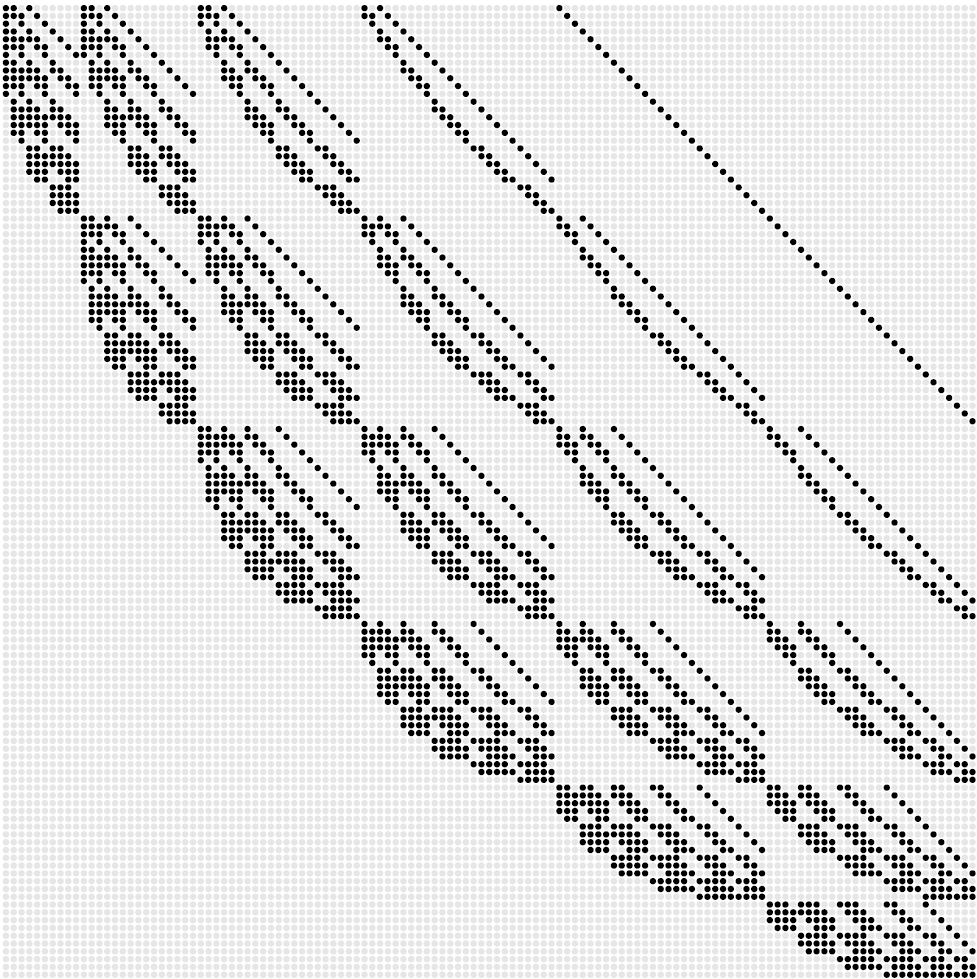}
\caption{Multiplication matrix when $n=4, k=4$ and $d=6$.}
\label{fig:Matrix_4,4,6}
\end{figure}
Finally, we would like to mention some connections to combinatorial design. Consider a fixed set $X$ with $n$ elements. We can then construct a matrix $W$ with rows indexed by all subsets $T$ of $X$ having $i$ elements and columns indexed by all subsets $K$ of $X$ having $i+k$ elements where the entry $(T,K)$ of $W$ is one if $T\subseteq K$ and zero otherwise. When $d=2$, a matrix representing multiplication by $h_k$ on $\mathbf{k}[x_1,\dots, x_n]/(x_1^2, \dots, x_n^2)$ from degree $i$ to degree $i+k$ is then the same as the matrix $W$. The matrices $W$ are well understood and have been studied extensively, see for example \cite{Wilson} and the references therein. 
If instead $X$ is a multiset, then one can consider the same construction, looking at matrices $W$ representing inclusion of multisubsets of $X$ of two different sizes. In the case that $X$ is the multiset containing the elements $1,\dots, n$, each repeated $d-1$ times, we get that multiplication by $h_k$ on $\mathbf{k}[x_1,\dots, x_n]/(x_1^d, \dots, x_n^d)$ from degree $i$ to $i+k$, can be represented by such a matrix $W$ defined by inclusion of multisubsets of size $i$ into multisubsets of size $i+k$. Despite the considerable amount of literature for the case when $X$ is a set, we have been unable to find any prior study on these matrices when $X$ is a multiset. 
When $n=2$, we know that they are Toeplitz matrices and have a formula for their determinants in Theorem \ref{thm:Toeplitz_det_formula}. For larger values of $n$, these matrices are no longer Toeplitz, but interesting patterns can still be found in them. See Figure \ref{fig:Matrix_4,4,6} and Figure \ref{fig:Matrix_7,6,3} for some illustrations of the matrices $W$ coming from the specified values of $n, k,$ and $d$. The value $i$ is chosen such that $W$ becomes a square matrix. Each black dot in the figure indicates a one in the matrix and each grey dot indicates a zero. 

\begin{figure}[h]
\includegraphics[scale=0.4]{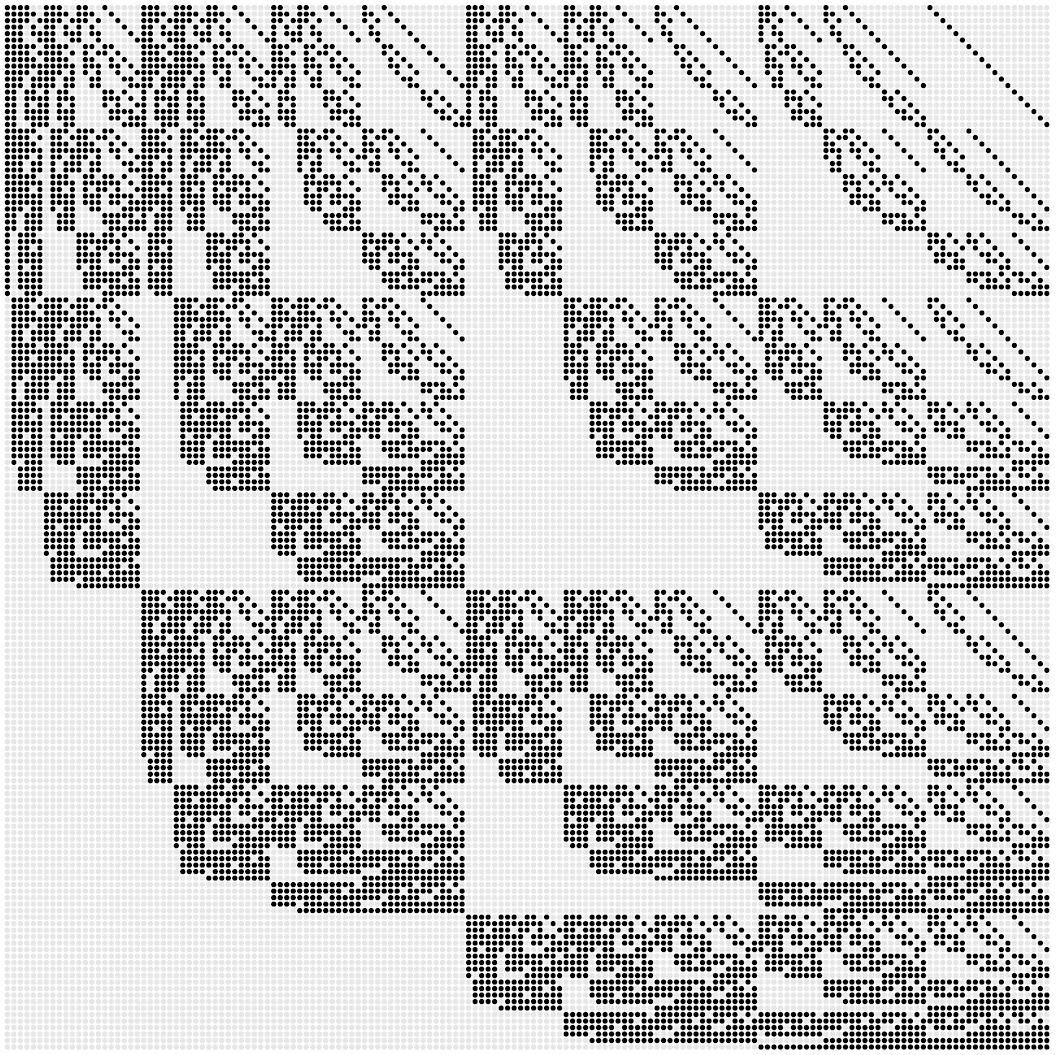}
\caption{Multiplication matrix when $n=7, k=6$ and $d=3$.}
\label{fig:Matrix_7,6,3}
\end{figure}

\section*{Acknowledgments}
Computer experiments in Macaulay2 were used extensively in the creation of this article. The authors would like to thank Per Alexandersson for valuable discussions on the theory of symmetric polynomials. The second author was funded by the Swedish Research Council VR2022-04009.

\bibliographystyle{plain}
\bibliography{references}

\appendix

\section{Artinian Gorenstein algebras}
We include here an appendix of some facts concerning multiplication by an element in an Artinian Gorenstein algebra that has been used earlier in the text. These properties should be well known to experts, but as we were unable to locate suitable references for them, we provide proofs of them here for completeness.

Let us begin by recalling that
an Artinian algebra $A=A_0\oplus A_1 \oplus \cdots \oplus A_e$ over a field $\mathbf{k}$ is \emph{Gorenstein} if it has a $1$-dimensional socle. That is, if $\dim_{\mathbf{k}}(A_e)=1$ and $\{f \; | \; fx_i=0 \text{ for all }i\}=A_e$.
There are many other equivalent definitions of an Artinian Gorenstein algebra. To state some more of them, recall Macaulay's inverse system from the preliminaries. Let $R=\mathbf{k}[x_1,\ldots,x_n]$ and consider the dual $S=k[X_1,\ldots,X_n]$, where $x_i$ acts on $F \in S$ as $x_i \circ  F = \frac{\partial F}{\partial X_i}$. For such an $F\in S$, the \emph{annihilator} $\mathrm{Ann}(F)=\{g\in R \; | \; g \circ F = 0\}$ is an ideal in $R$. 

We now collect definitions of Gorenstein algebras that we will use later.

\begin{proposition}
Let $A$ be an Artinian algebra over a field $\mathbf{k}$. Then the following are equivalent.
\begin{enumerate}
\item A is Gorenstein with socle degree $e$.
\item Multiplication on $A$ induces a perfect pairing $A_i \times A_{e-i} \to A_e\simeq \mathbf{k}$ for all $i=0,1,\dots, e$.
\item There is a form $F$ of degree $e$ such that $A\simeq \mathbf{k}[x_1,\dots, x_n]/(\mathrm{Ann}(F))$.
\end{enumerate}
\end{proposition}

The form $F$ such that $A\simeq \mathbf{k}[x_1,\dots, x_n]/(\mathrm{Ann}(F))$ is called a \emph{Macaulay dual generator} of the algebra $A$. 

The symmetry of an Artinian Gorenstein algebra now makes it easier to check if something is a max-rank element on the algebra.

\begin{lemma}\label{lem:inj_sufficient}
Let $f$ be a form of degree $k$ and $A$ an Artinian Gorenstein algebra with socle degree $e$. Then the map $\cdot f:A_i \to A_{i+k}$ has the same rank as $\cdot f: A_{e-(i+k)} \to A_{e-i}$. Moreover, if multiplication by $f$ is an injective map from $A_j$ to $A_{j+k}$ for the largest $j$ such that $\dim_{\mathrm{k}}(A_j)\leq \dim_{\mathrm{k}}(A_{j+k})$, then $f$ is a max-rank element on $A$.
\end{lemma}

\begin{proof}
Since $A$ is a Gorenstein algebra, we know that the multiplication on $A$ gives a perfect pairing. That is, let $B_i=\{g_1,\dots, g_s\}$ and $\hat{B}_{e-i}=\{\hat{g}_1,\dots, \hat{g}_s\}$ be bases of $A_i$ and $A_{e-i}$ respectively and say that $F$ is a dual generator of $A$. Then $g_a\cdot \hat{g_b} = \lambda_{a,b}F$ for some constant $\lambda_{a,b}$ for any $a,b$ and the matrix $G=(\lambda_{a,b})_{a,b=1}^s$ is invertible. Similarly, let $H=(\lambda_{a,b}')_{a,b=1}^t$ be the maximal rank matrix coming from the perfect pairing between $B_{e-(i+k)}=\{h_1,\dots, h_t\}$ and $\hat{B}_{i+k}=\{\hat{h}_1,\dots, \hat{h}_t\}$, bases of $A_{e-(i+k)}$ and $A_{i+k}$ respectively.

Now, if $M$ is representing the map $\cdot f:A_i \to A_{i+k}$ with respect to the above bases, then the $j$:th column of $M$ is given by $(c_{1,j},\cdots, c_{t,j})^T$ where $fg_j = c_{1,j}\hat{h}_1+\cdots + c_{t,j}\hat{h}_t$. By definition of $H$, we also have that 
\[
\begin{pmatrix}
h_1fg_j \\
\vdots\\
h_tfg_j
\end{pmatrix} = H
\begin{pmatrix}
c_{1,j}F\\
\vdots \\
c_{t,j}F
\end{pmatrix}.
\]
Hence, using that $H$ is invertible and applying the above reasoning to all columns of $M$, we find that 
\[
MF=(c_{i,j}F)=H^{-1}
\begin{pmatrix}
fh_1g_1 & f h_1g_2 & \cdots & fh_1g_s\\
fh_2g_1 & f h_2g_2 & \cdots & fh_2g_s\\
\vdots & \vdots  & \ddots & \vdots\\
fh_tg_1 & f h_tg_2 & \cdots & fh_tg_s\\
\end{pmatrix}.
\]
Further, if we change our indexing from $i$ to $e-(i+k)$, then instead of a map from $A_i$ to $A_{i+k}$, we get one from $A_{e-(i+k)}$ to $A_{e-i}$, and from our set up, this case is handled exactly in the same way but with the roles of $g$ and $h$ swapped. Hence, if $M'$ is the matrix representing $\cdot f:A_{e-(i+k)} \to A_{e-i}$ in our chosen bases, then 
\[
M'F = G^{-1} \begin{pmatrix}
fh_1g_1 & f h_2g_1 & \cdots & fh_tg_1\\
fh_1g_2 & f h_2g_2 & \cdots & fh_tg_2\\
\vdots & \vdots  & \ddots & \vdots\\
fh_1g_s & f h_2g_s & \cdots & fh_tg_s\\
\end{pmatrix}
=G^{-1}(HMF)^T = G^{-1}M^TH^TF.
\]
Since multiplication by invertible matrices does not change the rank and the rank of the transpose of a matrix is the same as the rank of the matrix, we conclude that $\rank(M')=\rank(G^{-1}M^TH^T)=\rank(M)$ and that the map $\cdot f:A_i \to A_{i+k}$ has the same rank as $\cdot f: A_{e-(i+k)} \to A_{e-i}$.

For the final part of the lemma we now argue as follows. Since $\dim_{\mathrm{k}}(A_i)=\dim_{\mathrm{k}}(A_{e-i})$ for all $i$, the first part gives that it suffices to establish maximal rank in all degrees where injectivity is required. Now, let $j$ be the largest degree such that $\cdot f: A_j \to A_{j+k}$ is injective. The by the first part we have that $\cdot f: A_{e-(j+k)} \to A_{e-j}$ is surjective, so $(A/(f))_{e-j}=0$. Since all variables have degree one, this forces that $(A/(f))_{e-j+l}=0$ for all $l\geq 0$, giving that $\cdot f:A_{e-j+l-d}\to A_{e-j+l}$ is surjective for all $l\geq 0$. A final application of the first part then gives that $\cdot f:A_{j-l} \to A_{j-l+d}$ is injective for all $l\geq 0$ and hence $f$ is injective in all degrees where injectivity can be required.
\end{proof}

A well known fact about Gorenstein algebras is that their Hilbert series are symmetric. That is, if $A$ is an Artinian Gorenstein algebra with socle degree $e$ and $\HS(A;t)=\sum_{i=0}^ea_it^i$ is the Hilbert series of $A$, then $a_i=a_{e-i}$ for all $i$. We would now like to extend this notion of symmetry to series where the constant term might vanish. To this end, when considering a series $p(t)=a_nt^n + a_{n-1}t^{n-1} + \cdots + a_mt^m$, where $a_na_m\neq 0$, we say that $p$ is \emph{symmetric} if $a_{m+i}=a_{n-i}$ for all $i$. 

Another property of Hilbert series that has gathered a lot of interest is questions about unimodality. The series $\HS(A;t)=\sum_{i=0}^ea_it^i$ is \emph{unimodal} if the is an integer $s$ such that $$a_0\leq a_1 \leq \cdots \leq a_{s-1} \leq a_s \geq a_{s+1} \geq \cdots \geq a_{e-1} \geq a_e.$$
Note that not all Gorenstein algebras have unimodal Hilbert series, but all complete intersections do. It turns out that if $A$ is a Gorenstein algebra with unimodal Hilbert series, then the amount by which a form fails to be a max-rank element on $A$ always defines a symmetric series.

\begin{proposition}\label{prop:Failing_symmetric}
Let $A$ be an Artinian Gorenstein algebra with unimodal Hilbert series, $f$ a form of degree $k$, and consider the polynomial
\[
[(1-t^k) \HS(A;t)]
\]
where the brackets denotes truncation of the polynomial at the first non-positive coefficient.
Then the difference $\HS(A/(f);t)-[(1-t^k) \HS(A;t)]$ is symmetric.
\end{proposition}

\begin{proof}
Say $A$ has socle degree $e$. We can assume that $k\leq e$ as otherwise the difference is the zero polynomial. Now, consider the exact sequences
\[
0 \to K_i \to A_i \xrightarrow{\cdot f} A_{i+k} \to (A/(f))_{i+k} \to 0 
\]
and
\[
0 \to K_{e-(i+k)} \to A_{e-(i+k)} \xrightarrow{\cdot f} A_{e-i} \to (A/(f))_{e-i} \to 0. 
\]
By Lemma \ref{lem:inj_sufficient}, we know that the map $\cdot f$ has the same rank in both cases, hence
\begin{equation}\label{eq:same_rank_eq}
\dim_{\mathbf{k}}(A_i) - \dim_{\mathbf{k}}(K_i) = \dim_{\mathbf{k}}(A_{e-(i+k)}) - \dim_{\mathbf{k}}(K_{e-(i+k)}).
\end{equation} 
Denote by $p(t)=\HS(A/(f);t)-[(1-t^k) \HS(A;t)]$. We are done if we can show that the coefficient of $t^{i+k}$ equals the coefficient of $t^{e-i}$ in $p$ for all $i$ satisfying that $i+k\leq e-i$. Next, using the short exact sequences, the coefficient of $t^{i+k}$ in $p(t)$ is
\begin{align*}
&\dim_{\mathbf{k}}(A/(f))_{i+k} - (\dim_{\mathbf{k}}(A_{i+k}) - \dim_{\mathbf{k}}(A_{i}))\\ 
=& (\dim_{\mathbf{k}}(A_{i+k}) - \dim_{\mathbf{k}}(A_{i}) + \dim_{\mathbf{k}}(K_i)) - (\dim_{\mathbf{k}}(A_{i+k}) - \dim_{\mathbf{k}}(A_{i})) \\
=& \dim_{\mathbf{k}}(K_i).
\end{align*}
Here we used the unimodality of the Hilbert series of $A$ to know the coefficient of $t^{i+k}$ in $[(1-t^k) \HS(A;t)]$. By using the short exact sequences again and \eqref{eq:same_rank_eq}, the coefficient of $t^{e-i}$ in $p(t)$ is
\begin{align*}
&\dim_{\mathbf{k}}(A/(f))_{e-i} - 0 \\
=& \dim_{\mathbf{k}}(A_{e-i}) - \dim_{\mathbf{k}}(A_{e-(i+k)}) + \dim_{\mathbf{k}}(K_{e-(i+k)}) \\
=& \dim_{\mathbf{k}}(A_{e-i}) -(\dim_{\mathbf{k}}(A_i) - \dim_{\mathbf{k}}(K_i)) \\
=&\dim_{\mathbf{k}}(K_i),
\end{align*}
where, in the final equality, we used that $\dim_{\mathbf{k}}(A_i)=\dim_{\mathbf{k}}(A_{e-i})$ since $A$ is an Artinian Gorenstein algebra with socle degree $e$. Hence the coefficients of $t^{i+k}$ and $t^{e-i}$ in $p(t)$ are equal for any $i$, showing that $p(t)$ is symmetric.
\end{proof}

\end{document}